%% file: squco.tex
\documentclass[12pt]{article}

\newcommand*{\diam}{\mathrm{diam}}
\newcommand*{\radius}{\mathrm{rad}}

\usepackage{amsmath, amsthm,enumerate,cite,color}
\usepackage{amssymb,graphicx}
\usepackage{fullpage}

\newtheorem{thm}{Theorem}[section]

\newtheorem{prop}[thm]{Proposition}
\newtheorem{lem}[thm]{Lemma}

\title{Searching for square-complementary graphs:\\ non-existence results and complexity of recognition}

\author{Ratko Darda\\
\small Institut de Math\'ematiques de Jussieu -- Paris Rive Gauche (IMJ-PRG)\\
 \small  Universit\'e Paris Diderot, France\\
\small \texttt{ratko.darda@imj-prg.fr}\\
\and
Martin Milani\v c\\
\small University of Primorska, IAM, Muzejski trg 2, SI6000 Koper, Slovenia\\
\small University of Primorska, FAMNIT, Glagolja\v ska 8, SI6000 Koper, Slovenia\\
\small \texttt{martin.milanic@upr.si}
\and
Miguel Piza{\~n}a\\
\small Universidad Aut\'onoma Metropolitana, 09340 Mexico City, Mexico\\
\small \texttt{map@xanum.uam.mx}\\
\small \texttt{http://xamanek.izt.uam.mx/map}
}

\date{\today}

\begin{document}
\maketitle
\begin{abstract}
A graph is {\it square-complementary} ({\it squco}, for short) if its
square and complement are isomorphic. We prove that there are no squco
graphs with girth 6, that every bipartite graph is an induced subgraph
of a squco bipartite graph, that the problem of recognizing squco
graphs is graph isomorphism complete, and that no nontrivial squco
graph is both bipartite and planar. These results resolve three of the
open problems posed in Discrete Math. {\bf 327} (2014) 62--75.

\medskip
\noindent{\bf Keywords:} graph equation; squco graph; complement; square
\end{abstract}

\section{Introduction} \label{sec:introduction}

Given two graphs $G$ and $H$, we say that $G$ is the {\it square} of
$H$ (and denote this by $G = H^2$) if their vertex sets coincide and
two distinct vertices $x$, $y$ are adjacent in $G$ if and only if $x$,
$y$ are at distance at most two in $H$. Squares of graphs and their
properties are well-studied in literature (see, e.g., Section 10.6 in
the monograph~\cite{MR1686154}). A graph $G$ is said to be {\it
  square-complementary} ({\it squco} for short) if its square is
isomorphic to its complement. That is, $G^2\cong \overline{G}$, or,
equivalently, $G\cong \overline{G^2}$. The terminology
``square-complementary'' (``squco'') was suggested
in~\cite{MilPedPelVer}, however the problem of characterizing squco
graphs is much older; it was posed by Seymour Schuster at a conference
in 1980~\cite{Schuster}. Since then, squco graphs were studied in the
context of graph equations, which may in general involve a variety of
operators including the line graph and complement, see,
e.g.,~\cite{MR613401, MR1322267, MR1405844, MR725876, MR1018613,
  MR1379114}. The entire set of solutions of some of these equations
was found (see for example \cite{MR613401} and references quoted
therein). The set of solutions of the equation $G^2 \cong
\overline{G}$ remains unknown, despite several attempts to describe it
(see for example \cite{MR1322267, MR1405844, MilPedPelVer}). The
problem of determining all squco graphs was also posed as Open Problem
No.~36 in Prisner's book~\cite{MR1379114}.

Examples of squco graphs are $K_1$, $C_7$, and a cubic
vertex-transitive bipartite squco graph on $12$ vertices, known as the
Franklin graph (see Fig.~\ref{fig:FranklinsGraph}).

\begin{sloppypar}
Every nontrivial squco graph has diameter $3$ or $4$~\cite{MR1405844},
but it is not known whether a squco graph of diameter $4$ actually
exists. In~\cite{MilPedPelVer}, several other questions regarding
squco graphs were posed, and a summary of the known necessary
conditions for squco graphs was given.  Among them (see
Proposition~\ref{prop:girth-7}), it was proved that the $7$-cycle is
the only squco graph of girth\footnote{The {\it girth} of
  a graph $G$ is the length of a shortest cycle in $G$, or $\infty$ if
  $G$ is acyclic.} at least $7$. This result leaves only $5$ possible
values for the girth $g$ of a squco graph $G$, namely $g\in
\{3,4,5,6,7\}$.  The case $g = 7$ is completely characterized by
Proposition~\ref{prop:girth-7}.  Balti\'c et al.~\cite{MR1322267} and
Capobianco and Kim~\cite{MR1405844} asked whether there exists a squco
graph of girth $3$.  An affirmative answer to this question was
provided in~\cite{MilPedPelVer} by a squco graph on $41$ vertices with
a triangle (namely, the circulant $C_{41}(\{4, 5, 8, 10\})$). As shown
by the Franklin graph, there also exists a squco graph of girth $4$.
The questions regarding the existence of squco graphs of girth $5$ or
$6$ were left as open questions in~\cite{MilPedPelVer}.  In
Section~\ref{sec:girth}, we answer one of them, namely
Open~Problem~3~in~\cite{MilPedPelVer}, by proving that there is no
squco graph of girth $6$.  This leaves $g = 5$ as the only possible
value of $g$ for which the existence of a squco graph of girth $g$ is
unknown.
\end{sloppypar}

The {\it bipartite complement} of a bipartite graph $G$ with
bipartition $\{A,B\}$ is the graph $\overline G^{\textrm{bip}}$
obtained from $G$ by replacing $E(G)$ by $AB \setminus E(G)$ (where
$AB$ denotes the set of all pairs consisting of a vertex in $A$ and a
vertex in $B$). A bipartite graph $G$ with bipartition $\{A,B\}$ is
said to be {\it bipartite self-complementary} if it is isomorphic to
its bipartite complement.

Bipartite squco graphs were characterized in~\cite{MilPedPelVer} as
the bipartite graphs that are bipartite self-complementary and of
diameter 3 (see Theorem~\ref{thm:bipartite}). This easily produces
infinitely many bipartite squco graphs as shown in
Section~\ref{sec:construction} or as already shown in
\cite[Theorem 5.7]{MilPedPelVer}.

An infinite family of planar squco graphs is given by
$C_7[k,1,1,1,1,1,1]$ for $k\ge 1$, where $C_7[k,1,1,1,1,1,1]$ denotes
the graph obtained from $C_7$ by replacing one vertex $v$ of $C_7$
with an independent set of $k$ new vertices and joining these $k$ new
vertices with an edge precisely to the vertices adjacent to $v$ in
$C_7$.

In Section~\ref{sec:planar}, we prove that no nontrivial squco graph
is both bipartite and planar.

Furthermore, In Section~\ref{sec:construction} we prove that every
bipartite graph is an induced subgraph of a bipartite squco
graph. This implies that squco graphs can contain arbitrarily long
induced paths and cycles, thus solving in particular Open Problem 8(2)
in \cite{MilPedPelVer} (which asked whether squco graphs can contain
arbitrarily long induced paths). We also show that the problem of
recognizing squco graphs is graph isomorphism complete, which solves
Open Problem 10 in \cite{MilPedPelVer}.

\section{Preliminaries} \label{sec:preliminaries}

We use standard graph terminology~\cite{MR2744811}. We briefly recall
some useful definitions.  Given two vertices $u$ and $v$ in a
connected graph $G$, we denote by $d_G(u,v)$ the distance in $G$
between $u$ and $v$ (that is, the number of edges in a shortest
$u$,$v$-path). For a positive integer $i$, we denote by $N_i(v,G)$
the set of all vertices $u$ in $G$ such that $d_G(u,v)= i$, and by
$N_{\ge i}(v,G)$ the set of all vertices $u$ in $G$ such that
$d_G(u,v)\ge i$.

For the reader's convenience, we transcribe here the known results
that we shall need.

\begin{prop}\label{prop:cut-vertex}
\textup{\cite{MR1322267}} Every squco graph is connected and has no
cut vertices.
\end{prop}

\begin{prop}\label{prop:radius-diameter}
\textup{\cite{MR1322267,MR1405844}} If $G$ is a nontrivial squco graph,
then we have $\radius(G) = 3$ and  $3\le \diam(G)\le 4\,.$ Moreover, if $G$ is
regular, then $\diam(G) = 3$.
\end{prop}

\begin{prop}\label{prop:girth-7}
\textup{\cite[Proposition 3.6]{MilPedPelVer}} If $G$ is a nontrivial
squco graph with girth at least~$7$, then $G$ is the $7$-cycle.
\end{prop}

\begin{prop}\label{thm:8.4.2}
\textup{\cite[Proposition 4.2]{MilPedPelVer}}
The only non-trivial squco graph with maximum degree at most 2 is $C_7$.
\end{prop}

A graph is \emph{subcubic} if it has maximum degree at most 3.

\begin{prop}\label{thm:8.4.7}
\textup{\cite[Proposition 4.7]{MilPedPelVer}}
The only subcubic squco graph on~12 vertices is the Franklin graph
(see Fig.~\ref{fig:FranklinsGraph}).
\end{prop}
\begin{figure}[!ht]
  \begin{center}
\includegraphics[height=30mm]{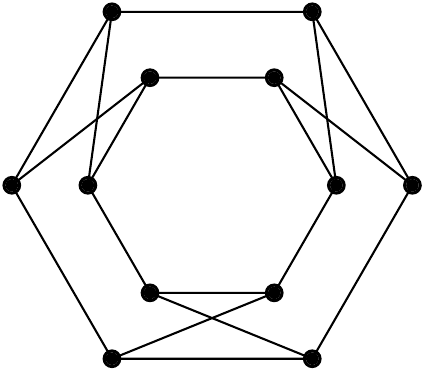}
  \end{center}
  \caption{The Franklin graph.}
\label{fig:FranklinsGraph}
\end{figure}

Given a graph $G$ with vertices labeled $v_1, v_2, \ldots, v_n$ and
positive integers $k_1, k_2, \ldots, k_n$, we denote by $G[k_1, k_2,
  \ldots, k_n]$ the graph obtained from $G$ by replacing each vertex
$v_i$ of $G$ with a set $U_i$ of $k_i$ (new) vertices and joining
vertices $u_i\in U_i$ and $u_j\in U_j$ with an edge if and only if
$v_i$ and $v_j$ are adjacent in $G$.

\begin{thm}\label{thm:8.4.8}
\textup{\cite[Theorem 4.8]{MilPedPelVer}} Let $G$ be a graph with at
most 11 vertices. Then, $G$ is squco if and only if $G$ is one of the
following eight graphs:\vspace{0.2cm}\\
\hspace*{0.5cm}$K_1$,\hspace{0.4cm}$C_7$,\hspace{0.4cm}$C_7[2,1,1,1,1,1,1]$,\hspace{0.4cm}$C_7[3,1,1,1,1,1,1]$,\hspace{0.4cm}$C_7[4,1,1,1,1,1,1]$,\vspace{0.2cm}\\
 \hspace*{1.2cm}$C_7[1,2,1,2,2,1,1]$,\hspace{0.4cm}$C_7[5,1,1,1,1,1,1]$,\hspace{0.4cm}$C_7[2,1,1,2,1,2,2]$.
\end{thm}

\begin{lem}\label{lem:bipartite1}
\textup{\cite[Lemma 5.1]{MilPedPelVer}} Let $G$ be a bipartite squco
graph with bipartition $\{A,B\}$.  Then, every two vertices in $A$
have a common neighbor (in $B$).
\end{lem}

Note every bipartite graph $G$ with $\diam(G)= 3$ satisfies
$\overline{G^2}=\overline{G}^{\textrm{bip}}$. This implies the following theorem.

\begin{sloppypar}
\begin{thm}\label{thm:bipartite}
\textup{\cite[Theorem 5.2]{MilPedPelVer}}
For a nontrivial bipartite graph $G$, the following conditions are equivalent:
\begin{enumerate}
  \item $G$ is squco.
  \item $G$ is bipartite self-complementary and of diameter $3$.
\end{enumerate}
\end{thm}
\end{sloppypar}

\section{The girth of a squco graph is not 6} \label{sec:girth}

\begin{thm}\label{GirthNot6}
There is no squco graph of girth $6$.
\end{thm}

\begin{proof}
Suppose for a contradiction that $G$ is a squco graph of girth $6$.
First, we observe that if $x$ is a vertex of $G$, then there are no
edges in any of the sets $N_i(x,G)$ for $i=1,2$ and that no two distinct
vertices in $N_1(x,G)$ have a common neighbor in $N_2(x,G)$.  Let $k =
\Delta(G)$ be the maximum degree of $G$, and let $w$ be a vertex of
degree $k$.  By Proposition~\ref{thm:8.4.2}, the only squco graphs with
maximum degree at most $2$ are $K_1$ and $C_7$, hence we have $k\ge
3$.

We consider two cases.

{\it Case 1.}  $w$ has a neighbor of degree at least three.

Let $v$ be a neighbor of $w$ of degree at least three, and let $p$ and
$q$ be two neighbors of $v$ other than $w$. If one of them, say $p$,
is of degree at least $3$, then $p$ has at least two neighbors in
$N_2(v,G)$ and thus $\Delta(\overline{G^2})\ge
|N_1(q,\overline{G^2})|\geq k+1$, contrary to the fact that
$\overline{G^2}\cong G$.  Since Proposition~\ref{prop:cut-vertex}
excludes the possibility of having degree~$1$ vertices, both $p$ and
$q$ are of degree~$2$. Let $a$ and $b$ be the unique neighbors of $p$
and $q$ in $N_2(v,G)$, respectively. The set $N_3(v,G)$ is nonempty,
because the radius of $G$ is $3$ by
Proposition~\ref{prop:radius-diameter}.  Vertices $a$ and $b$ must be
adjacent to all vertices in $N_3(v,G)$, otherwise
$\Delta(\overline{G^2})\ge
\max\{|N_1(p,\overline{G^2})|,|N_1(q,\overline{G^2})|\}\geq k+1$,
contrary to the fact that $\overline{G^2}\cong G$. To avoid a
$4$-cycle in $G$, we conclude that $|N_3(v,G)| = 1$.  But now, the
degree of $v$ in $\overline{G^2}$ is $1$, which implies that
$\overline{G^2}$ has a cut vertex, contrary to the fact that $G$ is
squco and Proposition~\ref{prop:cut-vertex}.

{\it Case 2.} All neighbors of $w$ are of degree at most two.

In this case, all neighbors of $w$ are of degree exactly two. In
particular, $|N_2(w,G)| = |N_1(w,G)|= k\ge 3$.  Now we will show that
every vertex $x$ from $N_2(w,G)$ is of degree at least $|N_3(w,G)|$.
Let $x\in N_2(w,G)$, and let $y$ be the unique neighbor of $x$ in
$N_1(w,G)$.  Vertex $x$ has at least $|N_3(w,G)|-1$ neighbors in
$N_3(w,G)$, since otherwise $|N_1(y,\overline{G^2})|\geq k+1$.  This
implies that any two vertices from $N_2(w,G)$ (the size of $N_2(w,G)$
is at least $3$) have at least $|N_3(w,G)|-2$ common neighbors in
$N_3(w,G)$. This bounds $|N_3(w,G)|\leq 3$, otherwise we would have a
$4$-cycle.

Also we have that $N_4(w,G)=\emptyset$ since, otherwise, for any
$z\in N_4(w,G)$ we would have $|N_1(z,\overline{G^2})|\geq k+1 >
\Delta(G) = \Delta(\overline{G^2})$, which is a contradiction. It
follows that $|G|=|\{w\}\cup N_1(w,G)\cup N_2(w,G)\cup N_3(w,G)|=
1+2k+|N_3(w,G)|$.

Suppose $|N_3(w,G)|=3$. To each of the three pairs of vertices in
$N_3(w,G)$, associate, if possible, their common neighbor in
$N_2(w,G)$. Because each vertex in $N_2(w,G)$ is connected to at least
two vertices in $N_3(w,G)$, it is surely associated with some pair.
If $|N_2(w,G)|\geq 4$ then some two vertices from $N_2(w,G)$ are
associated with the same pair and we get a $4$-cycle, a contradiction.
We thus have $|N_1(w,G)|=|N_2(w,G)|= k = 3$ and $|N_{\geq 4}(w,G)|=0$.
This implies that our graph has exactly ten vertices. All squco graphs
with at most $11$ vertices are listed in Theorem~\ref{thm:8.4.8} and
none of them has girth $6$. Hence this is a contradiction with $G$
having girth $6$.

Suppose $|N_3(w,G)|= 2$. If $k\leq 4$, then our graph has no more
than $11$ vertices, which is not possible.  Hence $k\geq 5$. There
must be at least $2k-1$ vertices of degree two in $G$ (all $k$
vertices in $N_1(w,G)$; at most one of the $k$ vertices in $N_2(w,G)$ has
both vertices from $N_3(w,G)$ for neighbors, otherwise we have a
$4$-cycle as before). In $\overline{G^2}$ at most $k+3$ of them are of
degree two, because every vertex in $N_1(w,G)$ will be connected to
all but one vertex in $N_2(w,G)$ in $\overline{G^2}$, which is a
contradiction, because $k\geq 5$.

The last possibility is that $|N_3(w,G)|=1$, but then $w$ would be of
degree $1$ in $\overline{G^2}$, again a contradiction. This completes
the proof.
\end{proof}

\section{A bipartite construction for squco graphs} \label{sec:construction}

Given two vertex-disjoint bipartite graphs $G$ and $G'$ with
respective bipartitions $\{A,B\}$ and $\{A',B'\}$, consider the
construction $H(G,G')$ as follows: Take four new vertices
$C=\{c_1,c_2\}$ and $D=\{d_1,d_2\}$ and construct a new bipartite
graph $H=H(G,G')$, with $V(H)=V(G)\cup V(G')\cup C\cup D$, with
bipartition $\{A\cup A'\cup C, B\cup B'\cup D\}$ and edges given by
$E(H)=E(G)\cup E(G')\cup \{c_1d_1, c_2d_2\}\cup AD\cup A'B\cup B'C$
(remember that $AD=\{ad :a\in A, d\in D\}$, and so on). For example,
the Franklin graph (see Fig.~\ref{fig:FranklinsGraph}) is the result
of this construction when each of $G$ and $G'$ is isomorphic to the
graph consisting of two disjoint copies of $K_2$.

Furthermore, given a bipartite graph $G$ with a bipartition $\{A,B\}$,
take again $C=\{c_1,c_2\}$, $D=\{d_1,d_2\}$ and construct the new
bipartite graph ${\it Ext}(G)$ such that $V({\it Ext}(G))=V(G)\cup
C\cup D$ and $E({\it Ext}(G))=E(G)\cup \{c_1d_1,c_2d_2\}\cup
\{d_1\}A\cup \{c_2\}B$. Consider the bipartition of ${\it Ext}(G)$ to
be $\{A\cup C,B\cup D\}$. Clearly $G$ is an induced subgraph of ${\it
  Ext}(G)$. Also, note that regardless of $G$, graphs ${\it Ext}(G)$
and $\overline{{\it Ext}(G)}^{\textrm{bip}}$ do not contain isolated
vertices and all the four parts involved (the two parts of ${\it
  Ext}(G)$ and the two parts of $\overline{{\it
    Ext}(G)}^{\textrm{bip}}$) are non-empty.

\begin{thm}\label{AnyBipInASquco}
Every bipartite graph is an induced subgraph of a bipartite squco
graph. In particular, there are squco graphs containing arbitrarily
long induced paths and cycles.
\end{thm}

\begin{proof}
Let $G$ be a bipartite graph with bipartition $\{A,B\}$ and
$G'=\overline{G}^{\textrm{bip}}$ with bipartition $\{A',B'\}$. Make $G$
and $G'$ disjoint. Thanks to the construction ${\it Ext}(G)$, we may assume
without loss of generality that no part of $G$ or $G'$ is empty and
that both $G$ and $G'$ have no isolated vertices.

Since $G'=\overline{G}^{\textrm{bip}}$, the graph $H=H(G,G')$ is bipartite
self-complementary: The required isomorphism, viewed as a permutation
of $V(H)$, leaves the subsets $C$ and $D$ fixed and exchanges $A$ with
$A'$ and $B$ with $B'$.

Since $A,A',B,B'\neq \emptyset$ and none of these sets contain
isolated vertices in $G$ or $G'$, it is easy to check that $H(G,G')$
has diameter $3$. It follows by Theorem~\ref{thm:bipartite} that
$H=H(G,G')$ is a bipartite squco graph containing $G$ as induced
subgraph.
\end{proof}

\begin{thm}\label{SqucoIsIsoComplete}
The problem of recognizing squco graphs is graph isomorphism complete.
\end{thm}

\begin{proof}
Given a graph $G$, consider $I(G)$, the vertex-edge bipartite
incidence graph of $G$, that is, the left part of $I(G)$ is $A=V(G)$,
the right part is $B=E(G)$, and $x\in A=V(G)$ is adjacent in $I(G)$ to
$e\in B=E(G)$, if and only if $x$ is incident to $e$ in $G$. Clearly
$G_1\cong G_2$ if and only if $I(G_1)\cong I(G_2)$. Without loss of
generality we can consider only graphs $G$ satisfying
\begin{enumerate}
\item $m =|E(G)| > |V(G)|= n>1$, and
\item $I(G)$ and $\overline{I(G)}^{\textrm{bip}}$ do not have isolated vertices.
\end{enumerate}
This is so because $G_1\cong G_2$ if and only if $G_1\ast K_2 \cong G_2\ast
K_2$, and because $G\ast K_2$ always satisfies conditions (1) and (2)
($G\ast K_2$ is the result of adding two universal vertices to $G$).
Also, it suffices to consider only pairs of graphs $G_1$, $G_2$ having
$|V(G_1)|= |V(G_2)|$ and $|E(G_1)|= |E(G_2)|$, since otherwise the
fact that $G_1$ and $G_2$ are not isomorphic can be detected in linear
time.

Now let $G_1$, $G_2$ be two such graphs, and $I_1:=I(G_1)$ and
$I_2:=\overline{ I(G_2)}^{\textrm{bip}}$. The rest of this proof is
devoted to showing that $G_1\cong G_2$ if and only if $H(I_1,I_2)$ is
squco.

If $G_1\cong G_2$, then $H(I_1,I_2)$ is squco because
of the considerations in the proof of Theorem~\ref{AnyBipInASquco}.

\begin{sloppypar}
Suppose now that $H(I_1,I_2)$ is squco. We shall show that the
required isomorphism (seen as a permutation of $H(I_1,I_2)$ )
exchanges $V(I_1)$ and $V(I_2)$ in the expected way ($A=V(G_1)$ with
$A'=V(G_2)$ and $B=E(G_1)$ with $B'=E(G_2)$) and hence, $I(G_1) \cong
I(G_2)$ which implies $G_1\cong G_2$.
\end{sloppypar}

First, observe that the left part of $H(I_1,I_2)$ has $2n+2$ vertices
and the right part contains $2m+2$ vertices. Since we assumed $m>n,$
it follows that any isomorphism must preserve each of the two
parts. Then observe that the degrees of the vertices in the right part
determine the subset to which they belong: vertices in subset $B'$
have degree $n$, vertices in subset $B$ have degree $n+2$ and vertices
in subset $D$ have degree $n+1$.  Then, also the subsets $A, B$ and
$C$ are uniquely determined.

It follows that the required isomorphism between $H(I_1,I_2)$ and
$\overline{H(I_1,I_2)^2}$ must exchange $V(I_1)$ and $V(I_2)$ in the
expected way and the result follows.
\end{proof}

\section{There are no nontrivial planar bipartite squco graphs} \label{sec:planar}

\begin{thm}\label{thm:bipartite-planar}
The only bipartite planar squco graph is $K_1$.
\end{thm}

\begin{proof}
Assume for a contradiction that $G$ is a bipartite planar squco graph
on $n$ vertices and $m$ edges, where $n>1$.  By
Theorem~\ref{thm:8.4.8}, $n\ge 12$, and by
Proposition~\ref{prop:cut-vertex}, $G$ is connected and of minimum
degree at least $2$.  Furthermore, the fact that $G$ is squco and
bipartite implies by Theorem~\ref{thm:bipartite} that $G\cong
\overline{G}^{\textrm{bip}}$.

Fix a bipartition $\{A,B\}$ of $G$ such that $a:=|A| \le b:=|B|$.  By
Theorem~\ref{thm:bipartite}, $G$ is bipartite self-complementary,
which implies that $|E(\overline G^{\textrm{bip}})| = |E(G)|$ and
consequently $m = ab/2$. Note that if $G$ is \emph{maximal} planar
bipartite, every face of any planar embedding of $G$ must be bounded
by a 4-cycle. In this case, the number of faces is $m/2$ and Euler's
formula yields $n-m+m/2 = 2$, or, equivalently, $m=2n-4$. It follows
that $m\leq 2n-4$ for any planar bipartite graph, and also, that when
$m=2n-5$ we must have that exactly one face of $G$ is bounded by a
6-cycle, while all the others are bounded by a 4-cycle.

From $m\le 2n-4$, we get that $ab\le 4n-8 = 4(a+b)-8$.
This yields
\begin{equation}\label{eq-ab}
(a-4)(b-4)\le 8\,.
\end{equation}
Since $a\le b$, we have $(a-4)^2\le 8$ (whenever $a\ge 4$) which
implies $a\le 6$.  We analyze several cases according to the value of
$a$.

\medskip
\noindent{\it Case 1. $a \le 3$.}  Since $G$ has no cut vertices,
every vertex of $B$ is adjacent to at least two vertices of $A$, which
means that in the bipartite complement of $G$, every vertex of $B$ is
adjacent to at most one vertex of $A$.  This means that $\overline
G^{\textrm{bip}}\cong G$ has a cut vertex, contradicting
Proposition~\ref{prop:cut-vertex}.

\medskip
\noindent{\it Case 2. $a = 4$.}  Since $G$ and $\overline
G^{\textrm{bip}}$ have no cut vertices, we infer that every vertex of
$B$ is adjacent to exactly two vertices of $A$.  Let $A =
\{v_1,v_2,v_3,v_4\}$. By Lemma~\ref{lem:bipartite1}, there exists a
vertex $u\in B$ such that $N(u) = \{v_1,v_2\}$, and a vertex $v\in B$
such that $N(v) = \{v_3,v_4\}$. Since $u$ and $v$ do not have a common
neighbor, Lemma~\ref{lem:bipartite1} yields a contradiction.

\medskip
\noindent{\it Case 3. $a = 5$.}  By Proposition~\ref{prop:cut-vertex},
neither $G$ nor $\overline G^{\textrm{bip}}$ have a vertex of degree
at most $1$, which implies that every vertex of $B$ is of degree
either two or three. If all vertices in $B$ are of degree two, we
obtain a contradiction as in Case 2.  So let $A = \{v_1, v_2, \ldots,
v_5\}$, and let $u\in B$ be a vertex of degree $3$, say $N(u)
=\{v_1,v_2,v_3\}$.  By Lemma~\ref{lem:bipartite1}, there is a vertex
in $B$, say $v$, such that $\{v_4,v_5\}\subseteq N(v)$.
But now, in the bipartite complement of $G$, vertices $u$ and $v$
have no common neighbor.
Hence Lemma~\ref{lem:bipartite1} implies that
$\overline G^{\textrm{bip}}$ is not squco, a contradiction.

\medskip
\noindent{\it Case 4. $a = 6$.}  Now, every vertex of $B$ is of degree
$2$, $3$, or $4$.  Let $A = \{v_1, v_2, \ldots, v_6\}$.  We claim that
every vertex of $B$ is of degree $3$. Indeed, if this were not the
case, then $B$ would contain a vertex $u$ of degree either $2$ or $4$.
If the degree of $u$ is $4$, say $N(u) = \{v_1, v_2, v_3, v_4\}$, then
any common neighbor $v$ of $v_5$ and $v_6$ (which exists by
Lemma~\ref{lem:bipartite1}) would not have any common neighbor with
$u$ in $\overline G^{\textrm{bip}}$, contradicting the fact that
$\overline G^{\textrm{bip}}$ (and hence $G$) is squco.  If the degree
of $u$ is $2$, then in the bipartite complement of $G$, vertex $u$
would be of degree $4$, and the above argument applies, leading to a
contradiction.

By~\eqref{eq-ab}, $b\le 8$.  If $b = 6$, then by the same argument as
above, every vertex of $A$ is of degree $3$, hence $G$ is cubic.  By
Proposition~\ref{thm:8.4.7}, $G$ is isomorphic to the Franklin graph,
see Figure~\ref{fig:FranklinsGraph}.  It is not hard to verify that
the Franklin graph contains a subdivision of $K_{3,3}$, hence it is
not planar.  The case $b = 6$ is thus impossible.

Suppose now that $b = 8$. We have $n = 14$, $m = 24$, which by Euler's
formula implies that every planar embedding of $G$ has exactly $12$
faces. Since in this case we have $m = 2n-4$, $G$ is maximal planar
bipartite and then, as discussed above, in every planar embedding of
$G$ every face is bounded by a $4$-cycle.  Fix a planar embedding of
$G$, and let $G^*$ be the corresponding dual graph.  Then, $G^*$ is a
$4$-regular plane graph with two types of faces: {\it $A$-faces}, that
is, the faces corresponding to vertices from $A$, and {\it $B$-faces},
that is, faces corresponding to vertices in $B$.  Since every vertex
of $B$ is of degree $3$ in $G$, every $B$-face is bounded by a
triangle.  Consider the graph $H$ whose vertices are the $B$-faces of
$G^*$, in which two distinct $B$-faces $f$ and $f'$ are adjacent if
and only if $f$ and $f'$ share a common vertex in $G^*$. Every
$B$-face $f$ is incident with exactly three vertices, each of which
yields a unique face adjacent with $f$ in $H$; since all these three
faces are distinct, $H$ is $3$-regular.  Moreover, the embedding of
$G^*$ gives rise to a natural planar embedding of $H$ in which there
is a bijective correspondence between $A$-faces of $G^*$ and faces of
$H$.  (In fact, $G^*$ is isomorphic to the medial graph of $H$, as
well as to its line graph but we will not need these facts here). By
Lemma~\ref{lem:bipartite1}, every two vertices in $A$ have a common
neighbor in $B$ (in $G$), therefore every two $A$-faces of $G^*$ are
adjacent with a common $B$-face of $G^*$, which, in terms of $H$,
means that every two faces of $H$ are incident with a common vertex.
Since $H$ is cubic, this implies that every two faces of $H$ are
incident with a common edge.  Consequently, the dual graph $H^*$ is a
complete graph of order $6$, which is impossible since $H^*$ is
planar.

We are left with the case $b = 7$. We represent this hypothetical
planar bipartite squco graph $G$ in Figure~\ref{fig:squcoplanar}(a)
and (c) (the thin graphs); vertices in $A$ are represented by squares
and vertices in $B$ are represented by triangles. In this case, $G$
has $n=a+b= 13$ vertices and $m=3b= 21$ edges.  We claim that every
vertex of $A$ is of degree $3$ or $4$. Indeed, if this were not the
case, then $A$ would contain a vertex $u$ of degree $2$ or $5$.  If
the degree of $u$ is $5$, then any common neighbor $v$ of the two
vertices in $B\setminus N(u)$ (which exists by
Lemma~\ref{lem:bipartite1}) would not have any common neighbor with
$u$ in $\overline G^{\textrm{bip}}$, contradicting the fact that
$\overline G^{\textrm{bip}}$ (and hence $G$) is squco.  If the degree
of $u$ is $2$, then in the bipartite complement of $G$ (which is
isomorphic to $G$ since $G$ is squco and by
Theorem~\ref{thm:bipartite}), $u$ would be of degree $2$, and the
above argument applies, leading again to a contradiction.

Since $G\cong \overline G^{\textrm{bip}}$ and the set $A$ is preserved
by any isomorphism from $G$ to its bipartite complement, $A$ contains
exactly $3$ vertices of degree $3$ and exactly $3$ vertices of degree
$4$.

Euler's formula implies that every planar embedding of $G$ has exactly
$21-13+2 =10$ faces and, since $m=2n-5$, $G$ is one edge away from
being maximal planar. Hence exactly one face is bounded by a
$6$-cycle, and each of the other $9$ faces is bounded by a $4$-cycle.
Fix a planar embedding of $G$, and let $G^*$ be the corresponding dual
graph (the thick graphs in Figure~\ref{fig:squcoplanar} (a) and (c)).
Then, $G^*$ is a simple plane graph with two types of faces: {\it
  $A$-faces}, that is, the faces corresponding to vertices from $A$,
and {\it $B$-faces}, that is, faces corresponding to vertices in $B$.
Every $A$-face is bounded by a triangle or by a $4$-cycle, and every
$B$-face is bounded by a triangle.  Graph $G^*$ has a unique vertex of
degree $6$, say $v^*$.  Let $v_1, v_2, \ldots, v_6$ be the cyclic
order of neighbors of $v^*$ in the planar embedding of $G^*$ such that
the face incident with vertices $v_1,v^*,v_2$ is an $A$-face.  Modify
the embedding by replacing the edges of the $2$-path $v_1v^*v_2$ by an
edge $v_1v_2$. This results in a $4$-regular plane multigraph $G'$
with $10$ vertices; its faces can be naturally partitioned into $6$
``$A$-faces'' and $6$ ``$B$-faces'' as in Figure~\ref{fig:squcoplanar}
(b) and (d).

\begin{figure}[ht!]
  \begin{center}
  \resizebox{0.8\textwidth}{!}{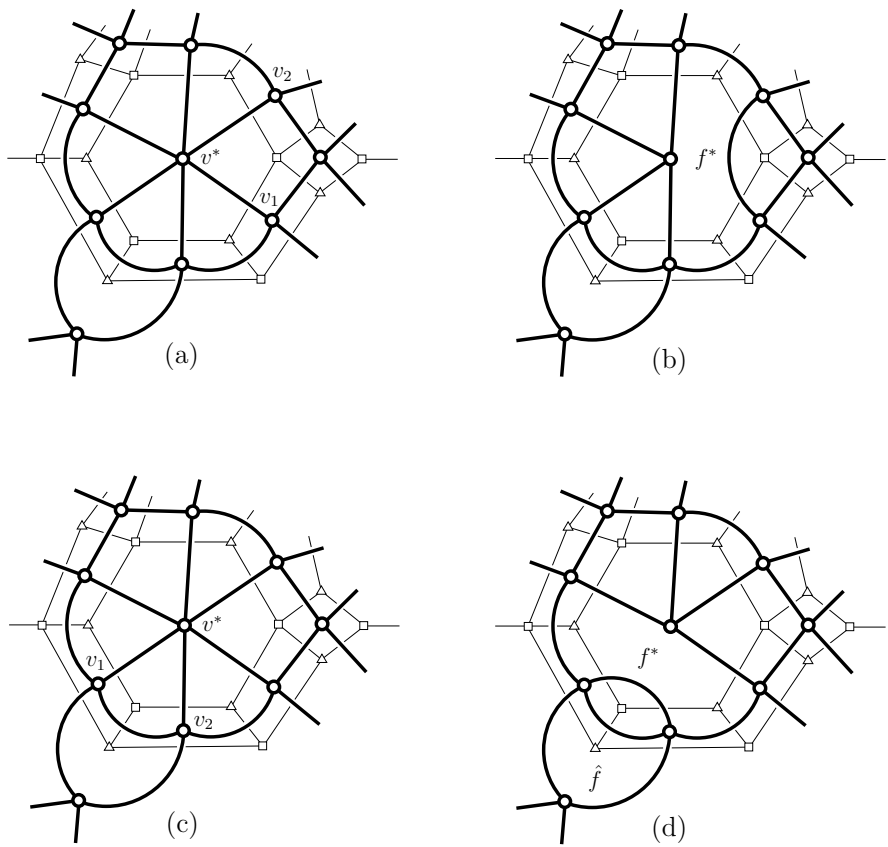}
  \end{center}
  \caption{(a) and (c) A hypothetical planar bipartite squco graph $G$
    with $a=6$ and $b=7$ and its dual $G^*$. (b) and (d) $G$ and the corresponding multigraph $G'$.}
\label{fig:squcoplanar}
\end{figure}

Since every vertex of $B$ is of degree $3$ in $G$, every $B$-face of
$G'$ is bounded by a triangle, except one, say $f^*$, which is bounded
by a $5$-cycle. Consider the multigraph $H$ the vertices of
which are the $B$-faces of $G'$, in which two distinct $B$-faces $f$
and $f'$ are adjacent with $k$ edges where $k$ is the number of
vertices of $G'$ incident to both $f$ and $f'$.  Every $B$-face $f$
other than $f^*$ is incident with exactly three vertices, each of
which yields a unique face adjacent with $f$ in $H$; face $f^*$ is of
degree $5$ in $H$; consequently $f^*$ is incident with either $4$ or
$5$ vertices. Moreover, the embedding of $G'$ gives rise to a natural
planar embedding of $H$ in which there is a bijective correspondence
between the set of $A$-faces of $G'$ and the set of faces of $H$.  By
construction, $G'$ has as many $A$-faces as $G^*$, with the only
difference that one of the $A$-faces has one edge less than the
corresponding $A$-face in $G^*$. It follows that every $A$-face in
$G'$ is bounded by a $k$-cycle for $k\in \{2,3,4\}$, and the same
holds for the faces of $H$.  In particular, no face of $H$ is bounded
by a $5$-cycle. Notice that $|V(H)| = 6$, $H$ has no loops, and $H$
has at most one pair of parallel edges, which, if they exist, are
incident with $f^*$, the vertex of degree $5$ in $H$.

Suppose first that $f^*$ is adjacent to all other vertices of $H$.
The remaining vertices are of degree $3$ in $H$ and hence $H-f^*$, the
graph obtained by deleting vertex $f^*$ from $H$, is a disjoint union
of cycles.  Since $H-f^*$ is a simple graph, it must be a single cycle
of length $5$. Since this cycle misses only one vertex, it is a facial
cycle in every planar embedding of $H$, which is impossible by the
above observation.  It is thus not possible that $f^*$ is adjacent to
all other vertices of $H$; in particular, together with another vertex
$\hat f$, it is part of a $2$-cycle (as in
Figure~\ref{fig:squcoplanar}(d)).  The graph $H-f^*$ is a simple graph
with degree sequence $(3,2,2,2,1)$.  There are only two such graphs: a
triangle with a pendant path of length $2$, and a $4$-cycle with a
pendant edge.  In the former case, every planar embedding of $H$
contains a face bounded by a $5$-cycle, which is impossible.  In the
latter case, every planar embedding of $H$ results in $G^*$ containing
two $A$-faces not adjacent to a common $B$-face of $G^*$, contrary to
the fact that in $G$, every two vertices in $A$ have a common neighbor
in $B$.

This completes the proof.
\end{proof}

\section{Conclusion}

\begin{sloppypar}
The results of the present paper provide further insight on the solution set of the graph equation $G^2\cong \overline{G}$. We showed that, while no solutions can be found among graphs of girth six or among nontrivial planar bipartite graphs, recognizing square-complementary graphs is in general as difficult as the Graph Isomorphism problem. This is a notorious problem that is not known to be either in P or NP-complete, and for which Babai recently announced a quasipolynomial-time algorithm~\cite{Babai}.
\end{sloppypar}

While our work answers three of the open problems on squco graphs posed in~\cite{MilPedPelVer}, several problems on squco graphs mentioned therein remain open. This includes existence of squco graphs within the classes of graphs of girth five, graphs of diameter four, and nontrivial chordal graphs. Furthermore, the proof of Theorem~\ref{SqucoIsIsoComplete} shows that the problem of recognizing squco graphs is graph isomorphism complete within the class of bipartite graphs. This motivates the study of the problem of recognizing squco graphs within subclasses of bipartite graphs for which the graph isomorphism problem is GI-complete, for example for chordal bipartite graphs~\cite{MR2112539}.

\medskip
\medskip
\noindent{\bf Acknowledgments.} This work is supported in part by the Slovenian Research Agency (I0-0035, research program P1-0285 and research projects N1-0032, J1-7051, J1-9110).


\def\soft#1{\leavevmode\setbox0=\hbox{h}\dimen7=\ht0\advance \dimen7
  by-1ex\relax\if t#1\relax\rlap{\raise.6\dimen7
  \hbox{\kern.3ex\char'47}}#1\relax\else\if T#1\relax
  \rlap{\raise.5\dimen7\hbox{\kern1.3ex\char'47}}#1\relax \else\if
  d#1\relax\rlap{\raise.5\dimen7\hbox{\kern.9ex \char'47}}#1\relax\else\if
  D#1\relax\rlap{\raise.5\dimen7 \hbox{\kern1.4ex\char'47}}#1\relax\else\if
  l#1\relax \rlap{\raise.5\dimen7\hbox{\kern.4ex\char'47}}#1\relax \else\if
  L#1\relax\rlap{\raise.5\dimen7\hbox{\kern.7ex
  \char'47}}#1\relax\else\message{accent \string\soft \space #1 not
  defined!}#1\relax\fi\fi\fi\fi\fi\fi}

\end{document}

%% file: dualg.pdf_t
\begin{picture}(0,0)%
\includegraphics{dualg.pdf}%
\end{picture}%
\setlength{\unitlength}{3947sp}%
\begingroup\makeatletter\ifx\SetFigFont\undefined%
\gdef\SetFigFont#1#2#3#4#5{%
  \reset@font\fontsize{#1}{#2pt}%
  \fontfamily{#3}\fontseries{#4}\fontshape{#5}%
  \selectfont}%
\fi\endgroup%
\begin{picture}(8584,8184)(417,-8034)
\put(2666,-5696){\makebox(0,0)[b]{\smash{{\SetFigFont{12}{14.4}{\familydefault}{\mddefault}{\updefault}{\color[rgb]{0,0,0}$v^*$}%
}}}}
\put(1659,-6001){\makebox(0,0)[b]{\smash{{\SetFigFont{12}{14.4}{\familydefault}{\mddefault}{\updefault}{\color[rgb]{0,0,0}$v_1$}%
}}}}
\put(2406,-7432){\makebox(0,0)[b]{\smash{{\SetFigFont{14}{16.8}{\familydefault}{\mddefault}{\updefault}{\color[rgb]{0,0,0}(c)}%
}}}}
\put(2399,-3377){\makebox(0,0)[b]{\smash{{\SetFigFont{14}{16.8}{\familydefault}{\mddefault}{\updefault}{\color[rgb]{0,0,0}(a)}%
}}}}
\put(6932,-1663){\makebox(0,0)[b]{\smash{{\SetFigFont{12}{14.4}{\familydefault}{\mddefault}{\updefault}{\color[rgb]{0,0,0}$f^*$}%
}}}}
\put(6613,-3424){\makebox(0,0)[b]{\smash{{\SetFigFont{14}{16.8}{\familydefault}{\mddefault}{\updefault}{\color[rgb]{0,0,0}(b)}%
}}}}
\put(5971,-7046){\makebox(0,0)[b]{\smash{{\SetFigFont{12}{14.4}{\familydefault}{\mddefault}{\updefault}{\color[rgb]{0,0,0}$\hat{f}$}%
}}}}
\put(2651,-1661){\makebox(0,0)[b]{\smash{{\SetFigFont{12}{14.4}{\familydefault}{\mddefault}{\updefault}{\color[rgb]{0,0,0}$v^*$}%
}}}}
\put(3270,-907){\makebox(0,0)[b]{\smash{{\SetFigFont{12}{14.4}{\familydefault}{\mddefault}{\updefault}{\color[rgb]{0,0,0}$v_2$}%
}}}}
\put(3158,-1992){\makebox(0,0)[b]{\smash{{\SetFigFont{12}{14.4}{\familydefault}{\mddefault}{\updefault}{\color[rgb]{0,0,0}$v_1$}%
}}}}
\put(6432,-5980){\makebox(0,0)[b]{\smash{{\SetFigFont{12}{14.4}{\familydefault}{\mddefault}{\updefault}{\color[rgb]{0,0,0}$f^*$}%
}}}}
\put(6613,-7467){\makebox(0,0)[b]{\smash{{\SetFigFont{14}{16.8}{\familydefault}{\mddefault}{\updefault}{\color[rgb]{0,0,0}(d)}%
}}}}
\put(2592,-6515){\makebox(0,0)[b]{\smash{{\SetFigFont{12}{14.4}{\familydefault}{\mddefault}{\updefault}{\color[rgb]{0,0,0}$v_2$}%
}}}}
\end{picture}%